\RequirePackage{plautopatch}
\documentclass[dvipdfmx,12pt]{amsart}

\usepackage[top=25truemm,bottom=20truemm,left=20truemm,right=20truemm]{geometry}
\usepackage[noadjust]{cite}
\usepackage[dvipsnames]{xcolor}
\usepackage{enumitem}
\usepackage[dvipdfmx]{graphicx} 
\usepackage{amsmath,amssymb,amsfonts,amssymb,amscd} 
\usepackage{amsthm} 
\usepackage{thmtools}
\usepackage[all]{xy} 
\usepackage{mathtools}
\usepackage[hidelinks]{hyperref}
\usepackage{cleveref} 

\theoremstyle{plain}
\newtheorem{thm}{Theorem}[section]
\newtheorem{cor}[thm]{Corollary}
\newtheorem{lem}[thm]{Lemma}
\newtheorem{prop}[thm]{Proposition}
\theoremstyle{definition}
\newtheorem{dfn}[thm]{Definition}
\newtheorem{ex}[thm]{Example}
\newtheorem{rmk}[thm]{Remark}
\crefname{thm}{Theorem}{Theorems}
\crefname{cor}{Corollary}{Corollarys}
\crefname{lem}{Lemma}{Lemmas}
\crefname{prop}{Proposition}{Propositions}
\crefname{dfn}{Definition}{Definitions}
\crefname{ex}{Example}{Examples}
\crefname{rmk}{Remark}{Remarks}
\crefname{section}{Section}{Sections}
\crefname{subsection}{Subsection}{Subsections}

\newcommand{\RR}{\mathbb{R}}

\newcommand{\ZZ}{\mathbb{Z}}
\newcommand{\qop}{\triangleleft}
\newcommand{\Qdl}{\mathbf{Qdl}}
\newcommand{\RGSet}{\mathbf{RGSet}}
\newcommand{\Inn}{\mathrm{Inn}}
\newcommand{\Dis}{\mathrm{Dis}}
\newcommand{\Conj}{\mathrm{Conj}}
\newcommand{\GAlex}{\mathrm{GAlex}}

\renewcommand{\SS}{\mathfrak{S}}

\date{\today}
\title{Quandles from group actions and a Cayley-type embedding theorem} 
\author{Ryoya Kai} 
\address{
	(R. Kai) 
  Center for Educational Research of Science and Mathematics,
	Nara University of Education, 
	Takabatake-Cho, Nara City, 
	Nara, 630-8301, Japan}
\email{kai.ryoya.d8@cc.nara-edu.ac.jp}

\begin{document}
\begin{abstract}
A quandle is an algebraic system that can be regarded as a generalization of the conjugation operation in groups. 
We study a quandle construction associated with group actions and determine its structural properties, including its inner automorphism group, connected components, and subquandles. 
As a principal application, we establish a Cayley-type embedding theorem for finite quandles. 
Applying the construction to the natural action of the symmetric group, we obtain, for each $n$, a single quandle into which every quandle of cardinality $n$ embeds.
\end{abstract}

\maketitle
\section{Introduction}

Quandles are algebraic systems introduced independently by Joyce \cite{Joyce-1982-ClassifyingInvariantKnotsKnota} and Matveev \cite{Matveev-1982-DistributiveGroupoidsKnotTheory} in the study of knots.
They also arise naturally in several other areas of mathematics.
Typical examples include conjugation in groups and point symmetries of symmetric spaces. 
Thus, quandle theory is closely related to topology, group theory and geometry, and it is natural to ask whether fundamental constructions and theorems in these areas admit quandle analogues.

One of the most fundamental results in group theory is Cayley's theorem,
which states that every finite group of order $n$ embeds into the symmetric group $\SS_n$. 
This naturally leads to the following question in quandle theory: 
can one construct, from the natural action of $S_n$ on $[n]:=\{1,\ldots,n\}$, 
a quandle into which every quandle of cardinality $n$ embeds?
The main purpose of this paper is to give an affirmative answer to this question.
The most immediate analogue of Cayley’s embedding is the natural representation $x \mapsto s_x$. 
However, this map need not be injective, since distinct elements of a quandle may induce the same point symmetry.
The construction below provides a way to retain the underlying element together with its symmetry and thereby overcome this obstruction.

To this end, we study a general construction of quandles from group actions. 
Let $\Phi: X\times G \to X$ be a right action of a group $G$ on a set $X$, 
and write $x^g:=\Phi(x,g)$ for $(x,g)\in X\times G$.
Then, the product $X \times G$ admits a quandle structure as follows:
\begin{dfn}[{\cref{lem:ActQdleStr} and \cref{def:ActQdle}}]
Define a binary operation on $X\times G$ by
\[
  (x,g) \qop_\Phi (y,h)
  :=
  \bigl(x^{g^{-1}h},h^{-1}gh\bigr).
\]
Then $(X\times G,\qop_\Phi)$ is a quandle, which we denote by
$X\times_\Phi G$.
\end{dfn}
This construction can also be obtained by applying the rack-to-quandle construction of Andruskiewitsch and Graña \cite{Andruskiewitsch-2003-RacksPointedHopfAlgebrasa} to the matrix racks of Fenn and Rourke \cite{Fenn-1992-RacksLinksCodimensionTwoa}; see \cref{rmk:MatRack_Kink}.
Although this construction is implicit in the preceding work, our aim is to develop its structure theory systematically. 
Specifically, we study its functoriality, inner automorphism group, connected components, subquandles, and relation to augmented quandles (see \cref{sec3}). 
More precisely, we explicitly determine the inner automorphism group of $X\times_\Phi G$ as a direct product involving a quotient of $G$ as follows.
\begin{thm}[\cref{cor:InnerAutomorphismGroup2}]
    For $m \in \ZZ$, define the map $\zeta_m: X \times_\Phi G \to X \times_\Phi G$ by $\zeta_m(x, g) := (x^{g^m},g)$.
    Then, we have an isomorphism 
    \[
        \Inn(X \times_\Phi G) \cong Z \times G/(Z(G) \cap K), 
    \] 
    where $Z = \{\zeta_m \mid m \in \ZZ\}$, $Z(G)$ is the center of $G$, and $K = \{g \in G \mid x^g = x \text{ for any $x \in X$}\}$ is the kernel of $\Phi$.
\end{thm}
In addition, the set of connected components, denoted by $\pi_0(X \times_\Phi G)$, is described in terms of double cosets involving stabilizers and centralizers.
\begin{thm}[\cref{thm:pi_0}]
    Let $B$ be a complete set of representatives for $X/G$, and let $R$ be a complete set of representatives for the conjugacy classes of $G$.
    Then there is a bijection 
    \[
    \pi_0(X \times_\Phi G) \cong \sqcup_{(b,r) \in B \times R)} C_G(r) \backslash G /H_b,
    \]
    where $H_b$ is the stabilizer at $b$ and $C_G(r)$ is the centralizer of $r$ in $G$.
\end{thm}
These results show that the construction retains substantial information from the underlying group action and, in particular, provides the framework for the embedding theorem below.

Applying this construction to the natural action $\Phi_n: [n]\times \SS_n \to [n]$ of the symmetric group $\SS_n$ on $[n]$ yields the quandle that serves as the target in our Cayley-type embedding theorem.
Define
\[
  Q\SS_n:= [n]\times_{\Phi_n} \SS_n.
\]
For quandles of cardinality $n$, the quandle $Q\SS_n$ plays a role analogous to that of $\SS_n$ for groups of order $n$. 
Our main embedding theorem is as follows.
\begin{thm}[Cayley-type embedding theorem; \cref{thm:CayEmbThm}]
\label{thm:intro-cayley}
Every quandle of cardinality $n$ admits an injective quandle
homomorphism
\[
  X\longrightarrow Q\SS_n.
\]
\end{thm}
Thus, for each $n$, the single quandle $Q\SS_n$ contains an isomorphic copy of every quandle of cardinality $n$.
In particular, this provides a uniform ambient object for all finite quandles of a fixed cardinality.
In this sense, $Q\SS_n$ provides a quandle-theoretic analogue of $\SS_n$ in Cayley's theorem. 
We also determine the basic structure of $Q\SS_n$ by applying the general results above. 
In particular, the connected components of $Q\SS_n$ admit a simple combinatorial description. 
More precisely, they are parametrized by pairs $(\lambda,d)$, where $\lambda$ is a partition of $n$ describing the cycle type of a permutation and $d$ is the length of the cycle containing the distinguished element of $[n]$.
See \cref{prop:conn_QSn} and \cref{cor:conn_QSn}.

The Cayley-type embedding theorem also yields a linear consequence.
Let $V$ be a finite-dimensional vector space over a field $\mathbb K$,
and let $\Psi_V\colon V\times GL(V)\to V$ be the natural right action. 
We define
\[
  GLQ(V):= V\times_{\Psi_V} GL(V).
\]
Note that $Conj(GL(V))$ can be identified with the subquandle $\{0\} \times GL(V)$ of $GLQ(V)$.
Homomorphisms into $GLQ(V)$ may be regarded as a generalization of ordinary linear representations of quandles, 
namely, quandle homomorphisms into $Conj(GL(V))$; see, for example, \cite{Elhamdadi-2018-FinitelyStableRacksRacka}.
Note that not every finite quandle can be embedded into a conjugation quandle; see, for example, \cite{Arai-2026-DetectingNonadmissibilityQuandlesColoringsa}.
On the other hand, combining \cref{thm:intro-cayley} with the permutation representation $\SS_n \to GL(\mathbb{K}^n)$, we obtain the following linear analogue.
\begin{thm}[\cref{thm:LinearRep}]
\label{thm:intro-glq}
Let $\mathbb K$ be any field. Every quandle of cardinality $n$ admits
an injective quandle homomorphism into $GLQ(\mathbb{K}^n)$.
\end{thm}

Therefore, although a finite quandle need not admit an injective homomorphism into $\Conj(GL(V))$,
every finite quandle admits an injective homomorphism into $GLQ(V)$ for some finite-dimensional vector space $V$. 
Thus, for each $n$, $GLQ(\mathbb K^n)$ provides a common linear ambient quandle for all quandles of cardinality $n$.

This paper is organized as follows. 
Section 2 recalls basic notions and examples of quandles. 
In Section 3, we develop the structure theory of quandles associated with group actions. 
Section 4 applies these results to the natural action of the symmetric group and proves the Cayley-type embedding theorem and its linear analogue.
\section{Review of quandles}
In this section, we review the notion of quandles.
For a non-empty set $X$ and a binary operation $\qop$ on $X$,
the pair $(X, \qop)$ is called a \emph{quandle} if the following conditions hold:
\begin{enumerate}
  \item[(Q1)] $x \qop x = x$ for any $x \in X$.
  \item[(Q2)] For any $y \in X$, the map $s_y: X \to X$ defined by $s_y(x) := x \qop y$ is bijective.
  \item[(Q3)] $(x \qop y) \qop z = (x \qop z) \qop (y \qop z)$ for any $x, y, z \in X$.
\end{enumerate}
When no confusion can arise, we simply write $X$ instead of $(X, \qop)$.
The map $s_y: X \to X$ is called the \emph{point symmetry} at $y \in X$.
We denote by $x \qop^n y := s_y^n(x)$.

A non-empty subset $A$ of a quandle $X$ is called a \emph{subquandle}
if $s_a(A) = A$ for every $a \in A$.
A subquandle inherits a quandle structure by restriction of the ambient binary operation $\qop$.

Let $(X, \qop_X)$ and $(Y, \qop_Y)$ be quandles.
A map $f: X \to Y$ is called a \emph{homomorphism} if it satisfies $f(x \qop_X x') = f(x) \qop _Y f(x')$ for any $x, x' \in X$.
A bijective homomorphism is called an \emph{isomorphism}.
Quandles are said to be \emph{isomorphic} if there exists an isomorphism between them.
An isomorphism from a quandle to itself is called an \emph{automorphism}.
By the axioms (Q2) and (Q3), the point symmetry $s_y: X \to X$ is an automorphism for any $y \in X$.

The set of automorphisms of $X$ is denoted by $\mathrm{Aut}(X)$.
We multiply automorphisms from left to right; namely, $fg := g \circ f$.
With this convention,
$\mathrm{Aut}(X)$ acts on $X$ from the right by $x^f := f(x)$, so that  $x^{fg} = (x^f)^g$.
We use the same convention for permutations.
The \emph{inner automorphism group} $\mathrm{Inn}(X)$ of $X$ is the subgroup of $\mathrm{Aut}(X)$ generated by $\{s_y \mid y \in X\}$.
The \emph{displacement group} $\Dis(X)$ of $X$ is the subgroup of $\Inn(X)$ generated by $\{s_x s_y^{-1} \mid x, y \in X\}$.
An orbit of the action of $\Inn(X)$ on $X$ is called a \emph{connected component},
and we denote the set of all connected components by $\pi_0(X)$.
A quandle is said to be \emph{connected} if the action of $\Inn(X)$ is transitive, that is, $\pi_0(X)$ consists of a single point.
A subset $B \subset X$ is called a \emph{base point set} if $B$ intersects each connected component in exactly one point.
In other words, a base point set is a complete set of representatives for the orbit space $X/\Inn(X)$.

In the rest of this section, we introduce some standard examples.

\begin{ex}\label{ex:Conj}
  Let $G$ be a group.
  We define a binary operation $\qop$ on $G$ by $g \qop h := \iota_h(g)$ for $g, h \in G$,
  where $\iota_h: G \to G$ is the inner automorphism defined by $\iota_h(g) := h^{-1}gh$.
  Then, the pair $(G, \qop)$ becomes a quandle called the \emph{conjugation quandle} of $G$, and is denoted by $\Conj(G)$.

  A subset $A \subset G$ with $\iota_a(A) = A$ for any $a \in A$ is a subquandle of $\Conj(G)$. 
  Such a subquandle is called a \emph{conjugation subquandle} and is denoted by $\Conj(A)$.
  Note that a conjugation subquandle $\Conj(A)$ of $\Conj(G)$ is a union of conjugacy classes of the subgroup generated by $A$.
\end{ex}

\begin{ex}\label{ex:GAlex}
  Let $\sigma$ be an automorphism of a group $G$.
  We define a binary operation $\qop$ on $G$ by $g \qop h := \sigma(gh^{-1}) h$ for $g, h \in G$.
  Then, the pair $(G, \qop)$ becomes a quandle called the \emph{generalized Alexander quandle}, and is denoted by $\GAlex(G, \sigma)$.
\end{ex}





\section{A quandle derived from a group action}
\label{sec3}

In this section, we construct a quandle from a group action and study its fundamental properties.

\subsection{Definition}
\label{subsec:definition}

Let $\Phi: X \times G \to X$ be a right action of a group $G$ on a set $X$,
and write $x^g := \Phi(x, g)$ for $x \in X$ and $g \in G$.
The action gives rise to a quandle structure on $X \times G$.

\begin{lem}\label{lem:ActQdleStr}
    Define the binary operation $\qop_\Phi$ on $X \times G$ by 
    \[
        (x, g) \qop_\Phi (y, h) := (x^{g^{-1}h}, h^{-1} g h) \quad \text{for }(x, g), (y,h) \in X \times G.
    \]
    Then, $(X \times G, \qop_\Phi)$ is a quandle.
\end{lem}
\begin{proof}
    For any $(x, g) \in X \times_\Phi G$, 
    we have 
    \begin{align*}
        (x, g) \qop_\Phi (x, g)
        = (x^{g^{-1} g}, g^{-1}g g)
        = (x, g),
    \end{align*}
    and hence axiom (Q1) holds.

    For $(y, h) \in X \times_\Phi G$, we define $t_{(y, h)}: X \times_\Phi G \to X \times_\Phi G$ by 
    $t_{(y, h)}(x, g) := (x^{gh^{-1}}, hgh^{-1})$ for $(x, g) \in X \times_\Phi G$.
    Then, $t_{(y,h)} = s_{(y,h)}^{-1}$. 
    Indeed, it satisfies 
    \begin{align*}
        s_{(y, h)}(t_{(y, h)}(x, g)) 
        = s_{(y, h)}(x^{gh^{-1}}, hgh^{-1})
        = (x^{gh^{-1} (hgh^{-1})^{-1} h}, h^{-1}(hgh^{-1})h)
        = (x, g),
    \end{align*}
    and
    \begin{align*}
        t_{(y, h)}(s_{(y, h)}(x, g)) 
        = t_{(y, h)}(x^{g^{-1}h}, h^{-1}gh)
        = (x^{g^{-1}h (h^{-1}gh) h^{-1}}, h(h^{-1}gh)h^{-1})
        = (x, g).
    \end{align*}
    Thus, axiom (Q2) holds.

    Let $(x, g), (y, h), (z, k) \in X \times_\Phi G$.
    Then, we have
    \begin{align*}
        ((x,g) \qop_\Phi (y,h)) \qop_\Phi (z,k)
        &= (x^{g^{-1}h}, h^{-1}gh) \qop_\Phi (z,k)\\
        &= (x^{g^{-1}h (h^{-1}gh)^{-1} k}, k^{-1}(h^{-1}gh)k)\\
        &= (x^{g^{-1}k (k^{-1}gk)^{-1} (k^{-1} h k)}, (k^{-1} h k)^{-1}(k^{-1}gk)(k^{-1} h k))\\
        &= (x^{g^{-1}k}, k^{-1}gk) \qop_\Phi (y^{h^{-1}k}, k^{-1} h k)\\
        &= ((x,g) \qop_\Phi (z,k)) \qop_\Phi ((y,h) \qop_\Phi (z,k)).
    \end{align*}
    This shows that axiom (Q3) holds, which completes the proof.
\end{proof}

\begin{dfn}\label{def:ActQdle}
    For a right action $\Phi: X \times G \to X$ of a group $G$ on a set $X$,
    we denote the quandle $(X \times G, \qop_\Phi)$ constructed in \cref{lem:ActQdleStr} by $X \times_\Phi G$.
\end{dfn}

\begin{rmk}\label{rmk:MatRack_Kink}
    Fenn and Rourke introduced an example of a rack arising from a group action, called the \emph{matrix rack} \cite[Section 7, Example 7]{Fenn-1992-RacksLinksCodimensionTwoa}.
    Andruskiewitsch and Graña \cite[\S 1.1.1]{Andruskiewitsch-2003-RacksPointedHopfAlgebrasa} described a construction that associates a quandle to a rack.
    Our construction coincides with the quandle obtained by applying this procedure to matrix racks.
\end{rmk}

\begin{rmk}\label{rmk:ProjExt}
    The projection $p_G: X \times_\Phi G \to \Conj(G)$ is a surjective quandle homomorphism. 
    In particular, since $|p^{-1}_G(g)| = |X|$ for every $g \in G$, 
    $p_G: X \times_\Phi G \to \Conj(G)$ is a quandle extension introduced in \cite{Andruskiewitsch-2003-RacksPointedHopfAlgebrasa}.
    More precisely, the dynamical cocycle 
    \[
    f: \Conj(G) \times \Conj(G) \to \mathrm{Map}(X\times X, X),\quad
    (g, h) \mapsto f_{g,h}
    \] associated with the extension $p_G: X \times_\Phi G \to \Conj(G)$ is given by $f_{g,h}(x,y) = (x^{g^{-1}h})$.
\end{rmk}

\begin{ex}\label{ex:TrivAct}
    If $\Phi: X \times G \to X$ is a trivial action of a group $G$ on a singleton $X=\{x\}$,
    then the quandle $X \times_\Phi G$ is isomorphic to the conjugation quandle $\Conj(G)$.
\end{ex}

Note that the point symmetry at $(x, g) \in X \times_\Phi G$ is equal to that at $(y, g) \in X \times_\Phi G$.
Thus, we denote $s_g := s_{(x, g)} = s_{(y, g)}$ for $x, y \in X$ and $g \in G$.
The following lemma is useful for computations.
\begin{lem}\label{lem:formula}
    For $(x_0, g_0), \dots, (x_k, g_k) \in X \times_\Phi G$ and $n_1, \dots, n_k \in \ZZ$, the following equality holds:
    \[
    (\cdots ((x_0, g_0) \qop_\Phi^{n_1} (x_1, g_1)) \qop_\Phi^{n_2} \cdots) \qop_\Phi^{n_k} (x_k, g_k) = (x_0^{\alpha}, \alpha^{-1}g_0\alpha),
    \]
    where $\alpha = g_0^{-(n_1 + \cdots + n_k)}g_1^{n_1} \cdots g_k^{n_k}$.
\end{lem}
\begin{proof}
For $k=1$, the formula follows by induction on $|n_1|$, 
using the formula for the inverse point symmetry in \cref{lem:ActQdleStr}.
The general case then follows by induction on $k$.
\end{proof}

In the rest of this subsection, we establish the functoriality of our construction.
First, we recall morphisms between right actions.
Let $\Phi: X \times G \to X$ and $\Psi: Y \times H \to Y$ be right actions.
For a map $f: X \to Y$ and a homomorphism $\rho: G \to H$ of groups,
a \emph{morphism} $\Phi \to \Psi$ is a pair $(f, \rho)$ with $f(x^g) = f(x)^{\rho(g)}$ for any $x \in X$ and $g \in G$.
A morphism $(f, \rho)$ is called an \emph{isomorphism} if the map $f: X \to Y$ is bijective, and the homomorphism $\rho: G \to H$ is an isomorphism.

\begin{prop}\label{prop:EquivariantHom}
    Let $\Phi: X \times G \to X$ and $\Psi: Y \times H \to Y$ be right actions.
    For a map $f: X \to Y$ and a homomorphism $\rho: G \to H$ of groups,
    if $(f, \rho)$ is a morphism $\Phi \to \Psi$, then the product $f \times \rho: X \times_\Phi G \to Y \times_\Psi H$ is a homomorphism.
\end{prop}
\begin{proof}
    For $(x, g), (y, h) \in X \times_\Phi G$, we have 
    \begin{align*}
        f \times \rho((x, g) \qop_\Phi (y, h)) 
        &= f \times \rho(x^{g^{-1}h}, h^{-1}gh)
        = (f(x^{g^{-1}h}), \rho(h^{-1}gh))\\
        &= (f(x)^{\rho(g)^{-1}\rho(h)}, \rho(h)^{-1}\rho(g)\rho(h))
        = (f(x), \rho(g)) \qop_\Psi (f(y), \rho(h)),
    \end{align*}
    as desired.
\end{proof}

\begin{rmk}\label{rmk:FuctorFromGset}
    Let $\RGSet$ be the category of right actions and morphisms.
    Our construction defines a functor from $\RGSet$ to $\Qdl$ by $(\Phi: X \times G \to X)\mapsto X \times_\Phi G$, where $\Qdl$ is the category of quandles and homomorphisms of quandles.
\end{rmk}

\subsection{Inner automorphism group}
\label{subsec:Innerauto}
We now turn to the internal structure of $X\times_\Phi G$, beginning with its inner automorphism group.
To describe the group, we introduce a special automorphism $\zeta_n: X \times_\Phi G \to X \times_\Phi G$.

\begin{lem}\label{lem:zeta_n}
    For $n \in \ZZ$, define the map $\zeta_n: X \times_\Phi G \to X \times_\Phi G$ by $\zeta_n(x, g) := (x^{g^n}, g)$ for $(x, g) \in X \times_\Phi G$.
    Then, the following hold:
    \begin{enumerate}
        \item[$(1)$] The map $\zeta_n$ is an automorphism of $X \times_\Phi G$.
        \item[$(2)$] The order of $\zeta_1$ is equal to the exponent of the permutation group induced by the action of $G$ on $X$, where the exponent is allowed to be infinite.
        \item[$(3)$] For any $g \in G$, we have $\zeta_n s_g = s_g \zeta_n$.
        \item[$(4)$] $s_{gh} = \zeta_{1} s_g s_h$ and $s_g^{-1} = s_{g^{-1}} \zeta_2$ for any $g, h \in G$. 
    \end{enumerate}
\end{lem}
\begin{proof}
    First, we show $(1)$.
    For $(x, g), (y, h) \in X \times_\Phi G$, we have
    \begin{align*}
        \zeta_n((x, g) \qop_\Phi (y, h)) 
        &= \zeta_n(x^{g^{-1}h}, h^{-1}gh) 
        = (x^{g^{-1}h (h^{-1}gh)^n}, h^{-1}gh) 
        = (x^{g^{n-1}h}, h^{-1}gh)
        = (x^{g^n}, g) \qop_\Phi (y^{h^n}, h)\\
        &= \zeta_n(x, g) \qop_\Phi \zeta_n(y, h).
    \end{align*}
    Thus, the map is a homomorphism.
    Since the inverse of $\zeta_n$ is equal to $\zeta_{-n}$, the map $\zeta_n$ is bijective.
    Thus, $\zeta_n$ is an isomorphism, as desired.

    Since $\zeta_1^N = \zeta_N$ for $N \in \ZZ$, we have $\zeta_1^N = \mathrm{id}$ if and only if $x^{g^N} = x$ for any $x \in X$ and $g \in G$.
    This shows $(2)$.

    Next, we show $(3)$.
    For any $(y, h) \in X \times_\Phi G$, we have
    \begin{align*}
        (y, h)^{\zeta_n s_g} 
        &= (y^{h^n}, h)^{s_g} 
        = ((y^{h^n})^{h^{-1}g}, g^{-1}hg)
        = (y^{h^{n-1}g}, g^{-1}hg)
        = ((y^{h^{-1}g})^{(g^{-1}hg)^n}, g^{-1}hg)\\
        &= (y^{h^{-1}g}, g^{-1}hg)^{\zeta_n}
        = (y, h)^{s_g \zeta_n}.
    \end{align*}
    Hence, we obtain $\zeta_n s_g = s_g \zeta_n$.

    Finally, we show $(4)$.
    For any $(z, a) \in X \times_\Phi G$, we have
    \begin{align*}
        (z, a)^{s_{gh}} 
        &= (z^{a^{-1}gh}, (gh)^{-1}agh)
        = ((z^{g})^{(g^{-1}ag)^{-1}h}, h^{-1}g^{-1}a gh)\\
        &= (z^g, g^{-1} a g)^{s_h}
        = ((z^{a})^{a^{-1}g},g^{-1} a g)^{s_h}
        = (z^a, a)^{s_g, s_h}
        = (z, a)^{\zeta_1 s_g s_h},
    \end{align*}
    and hence $s_{gh} = \zeta_1 s_g s_h$.
    In addition, we have
    \begin{align*}
        (z, a)^{s_g^{-1}}
        &= (z^{ag^{-1}}, g a g^{-1})
        = ((z^{a^{-1} g^{-1}})^{(gag^{-1})^{2}})
        = (z^{a^{-1} g^{-1}}, g a g^{-1})^{\zeta_2}
        = (z, a)^{s_{g^{-1}} \zeta_2}
    \end{align*}
    Thus, we obtain $s_g^{-1} = s_{g^{-1}} \zeta_2$.
    This completes the proof.
\end{proof}

By \cref{lem:zeta_n} $(4)$, we have the following corollary.

\begin{cor}\label{cor:InnerAutomorphismGroup1}
    We have $\Inn(X \times_\Phi G) = \{\zeta_n s_g \mid n \in \ZZ, g \in G\}$.
\end{cor}

It is easy to show that $s_{h^{-1}gh} = s_{h}^{-1} s_g s_h$ for any $g, h \in G$.
Thus, the map $s: \Conj(G) \to \Conj(\Inn(X \times_\Phi G))$ is a quandle homomorphism.
However, by \cref{lem:zeta_n} (4), it need not be a group homomorphism.

We define the subgroup $Z$ of $\Inn(X \times_\Phi G)$ by $Z := \{\zeta_n \mid n \in \ZZ\}$.
By \cref{lem:zeta_n}, this subgroup $Z$ lies in the center of $\Inn(X \times_\Phi G)$.

\begin{prop}\label{prop:InnerAutomorphismGroup}
    Let $\Phi: X \times G \to X$ be a right action,
    let $Z(G)$ be the center of $G$,
    and 
    let $K$ be the kernel of $\Phi$, that is,
    \[
        Z(G) = \{z \in G \mid zg = gz \text{ for any $g \in G$}\}, \quad
        K = \{g \in G \mid x^g = x \text{ for any $x \in X$}\}.
    \]
    Then, the map $F: \Inn(X \times_\Phi G) \to G/(Z(G) \cap K)$ defined by $F(\zeta_n s_g) := [g]$ for $n \in \ZZ$ and $g \in G$ is a surjective group homomorphism, and $\ker(F) = Z$.
\end{prop}

\begin{proof}
    First, we show that the map $F: \Inn(X \times_\Phi G) \to G/(Z(G) \cap K)$ is well-defined.
    To show this, we prove that $\zeta_m s_{g_1} = \zeta_n s_{g_2}$ implies $g_1g_2^{-1} \in Z(G) \cap K$ for $g_1, g_2 \in G$ and $m, n \in \ZZ$.
    Then, for any $(y, h) \in X \times_\Phi G$, we have 
    \[
        (y^{h^{m-1} g_1}, g_1^{-1}hg_1) 
        = (y, h)^{\zeta_m s_{g_1}} 
        = (y, h)^{\zeta_n s_{g_2}} = (y^{h^{n-1} g_2}, g_2^{-1}hg_2).
    \]
    From the second component, we obtain $g_1g_2^{-1} h = h g_1g_2^{-1}$ for any $h \in G$, and hence $g_1g_2^{-1} \in Z(G)$.
    From the first component, taking $h = 1$, we also obtain $y^{g_1} =  y^{g_2}$ for any $y \in X$.
    Thus, we have $g_1g_2^{-1} \in K$.
    Hence, we have $g_1g_2^{-1} \in Z(G) \cap K$.
    Therefore, we conclude that the map $F: \Inn(X \times_\Phi G) \to G/(Z(G) \cap K)$ is well-defined.

    It is clear that the map $F$ is surjective.
    By \cref{lem:zeta_n} $(4)$, the map $F$ is a group homomorphism.
    Finally, we show that $\ker(F) = Z$.
    The inclusion $Z \subset \ker(F)$ is clear.
    Conversely, let $\zeta_n s_g \in \ker(F)$.
    Then, $g \in Z(G) \cap K$, and we have
    \[
        (y,h)^{s_g} = (y^{h^{-1}g}, g^{-1}hg) = (y^{h^{-1}}, h) = (y, h)^{\zeta_{-1}}.
    \]
    Therefore, we obtain $\zeta_n s_g = \zeta_{n-1} \in Z$.
    This completes the proof.
\end{proof}

\begin{cor}\label{cor:InnerAutomorphismGroup2}
    The group $\Inn(X \times_\Phi G)/Z$ is isomorphic to the group $G/(Z(G) \cap K)$.
    In other words, we have the following short exact sequence:
    \begin{equation*}
        1 \to Z \to \Inn(X \times_\Phi G) \to G/(Z(G) \cap K) \to 1.
    \end{equation*}
    Moreover, this short exact sequence splits, and we have $\Inn(X \times_\Phi G) \cong Z \times G/(Z(G) \cap K)$.
\end{cor}
\begin{proof}
The isomorphism and the existence of the short exact sequence immediately follow from \cref{prop:InnerAutomorphismGroup}.
Define the map $i: G \to \Inn(X \times_\Phi G)$ by $i(g) := \zeta_1 s_g$.
Then, this map is a group homomorphism.
Indeed, for $g, h \in G$, 
by \cref{lem:zeta_n} we have 
   \begin{align*}
       i(gh) = \zeta_1s_{gh} = \zeta_1^2 s_{g} s_h = \zeta_1 s_g \zeta_1 s_h = i(g)i(h).
   \end{align*}
For any $g \in G$, we have 
\[
    (y, h)^{i(g)} = (y^g, g^{-1}hg).
\]
This shows that $\ker(i) = Z(G) \cap K$.
Thus, the map $i$ induces the injective group homomorphism $\bar{i}: G/(Z(G) \cap K) \to \Inn(X \times_\Phi G)$, and this map is a section of the map $F: \Inn(X \times_\Phi G) \to G/(Z(G) \cap K)$ defined in \cref{prop:InnerAutomorphismGroup}.
This shows the short exact sequence splits.
Since $Z$ lies in the center of $\Inn(X \times_\Phi G)$, we have $\Inn(X \times_\Phi G) \cong Z \rtimes G/(Z(G) \cap K) \cong Z \times G/(Z(G) \cap K)$.
This completes the proof.
\end{proof}

\subsection{Connected components}
\label{subsec:conn_comp}
Having determined the inner automorphism group, we next describe its orbits on $X\times_\Phi G$, that is, the connected components.
Let $\Psi: (X \times G) \times G \to X \times G$ be the diagonal action defined by $\Psi((x, g), h) := (x^h, h^{-1} g h)$.
We denote the orbit space of the diagonal action $\Psi: (X \times G) \times G \to X \times G$ by $(X \times G )/G$.

\begin{prop}\label{prop:conn_comp}
  The following hold:
  \begin{enumerate}
      \item[$(1)$] Elements $(x, g), (b, a) \in X \times_\Phi G$ belong to the same connected component if and only if there exists an element $h \in G$ such that $(x, g) = \Psi((b, a), h)$.

      \item[$(2)$] There exists a bijection $\pi_0(X \times_{\Phi} G) \to (X \times G )/G$.
  \end{enumerate}
\end{prop}

\begin{proof}
    We show $(1)$.
    First, we assume that $(x,g), (b,a) \in X \times_\Phi G$ belong to the same connected component.
    Then, there exist $g_1, \dots g_n \in G$ and $\varepsilon_1, \dots, \varepsilon_n \in \{\pm 1\}$ such that $(x, g) = (b, a)^{s_{g_1}^{\varepsilon_1}\cdots s_{g_n}^{\varepsilon_n}}$.
    Applying \cref{lem:formula}, we obtain $(x, g) = (b^\alpha, \alpha^{-1} a \alpha)$, where $\alpha := a^{-(\varepsilon_1+ \cdots + \varepsilon_n)} g_1^{\varepsilon_1} \cdots g_n^{\varepsilon_n} \in G$.
    This proves the necessity.
    Conversely, suppose that $(x, g) = (b^h, h^{-1}ah)$ for $h \in G$.
    Then, it satisfies \[
    (b, a)^{s_h \zeta_1} = (b^{a^{-1}h}, h^{-1} a h)^{\zeta_1} = (b^h, h^{-1} a h) = (x, g).
    \]
    Since the automorphism $\zeta_1$ of $X \times_\Phi G$ is contained in $\Inn(X \times_\Phi G)$ by \cref{lem:zeta_n} $(4)$,
    we conclude that $(x, g)$ and $(b,a)$ belong to the same connected component,
    as desired.

    Since the assertion $(1)$ shows that each connected component of $X \times_\Phi G$ is equal to an orbit of the action $\Psi$,
    we immediately obtain $(2)$.
    This completes the proof.
\end{proof}

Finally, we give a complete set of representatives for the orbit space $(X \times G)/G$.
Let $R \subset G$ be a complete set of representatives for the conjugacy classes of $G$, that is, 
\[
G = \sqcup_{r \in R} r^G, \quad \text{where } r^G = \{g^{-1}rg \mid g \in G\}.
\]
For each $g \in G$, we choose $r_g \in R$ and $u_g \in G$ which satisfy $g = u_g^{-1} r_g u_g$.
Note that for $g \in G$, the element $r_g$ is uniquely determined, but $u_g$ need not be.
We denote the centralizer of $g$ in $G$ by $C_G(g) := \{h \in G \mid gh = hg\}$.
Let $B$ be a complete set of representatives for the orbit space $X/G$ of $\Phi: X \times G \to X$. 
The orbit decomposition gives 
\[
    X \cong \sqcup_{b \in B} H_b \backslash G, 
    \quad 
    \text{where $H_b$ is the stabilizer at $b$.}
\]
For each $x \in X$, we choose $b_x \in B$ and $v_x \in G$ such that $x = b_x^{v_x}$.
Note that for $x \in X$, the element $b_x$ is uniquely determined, but $v_x$ need not be.

An orbit of the diagonal action determines an orbit of $X$ and a conjugacy class of $G$. 
After fixing representatives $b \in B$ and $r\in R$, the remaining ambiguity is given by the stabilizer $H_b$ and the centralizer $C_G(r)$, which naturally leads to a double-coset description as follows.

\begin{prop}\label{prop:representative}
    The map $\varphi: (X \times G)/G \to \sqcup_{(b, r) \in B \times R} C_G(r) \backslash G / H_b$ defined by $\varphi([x, g]) = C_G(r_g) u_g v_x^{-1} H_{b_x}$ is bijective.
\end{prop}

\begin{proof}
    First, we show that the map $\varphi$ is well-defined.
    Suppose $[x, g] = [y, h]$ in $(X \times G)/G$.
    Then, there exists $\alpha \in G$ such that $(y, h) = (x^\alpha, \alpha^{-1} g \alpha)$.
    This shows that $x$ and $y$ lie in the same orbit of the action $\Phi$, and $g$ and $h$ are conjugate in the group $G$.
    Thus, we have $b_x = b_y$ and $r_g = r_h$.
    Let $s := v_x \alpha v_y^{-1}$ and $c := u_g \alpha u_h^{-1}$.
    Then $s \in H_{b_x}$ and $c \in C_G(r_g)$. 
    Indeed,
    \begin{align*}
        b_x^s = b_x^{v_x \alpha v_y^{-1}} = x^{\alpha v_y^{-1}} = y^{v_y^{-1}} = b_y = b_x,
    \end{align*}
    and 
    \begin{align*}
        c r_g 
        = u_g \alpha u_h^{-1} r_h 
        = u_g \alpha h u_h^{-1} 
        = u_g g \alpha u_h^{-1}
        = r_g u_g \alpha u_h^{-1}
        = r_g c.
    \end{align*}
    Thus, we obtain
    \begin{align*}
        u_h v_y^{-1} 
        = (c^{-1} u_g \alpha) (\alpha^{-1} v_x^{-1} s) 
        = c^{-1} u_g  v_x^{-1} s
        \in C_G(r_g) u_g v_x^{-1} H_{b_x}. 
    \end{align*}
    This shows $\varphi([x, g]) = \varphi([y, h])$,
    and hence the map $\varphi$ is well-defined.

    Next, we show that the map $\varphi$ is surjective.
    Let $(b, r) \in B \times R$ and $g \in G$.
    Let $y := b$ and $h := g^{-1} rg$.
    Then, we have $b_y = b$, $v_y \in H_b$, $r_h = r$, and $u_h g^{-1} \in C_G(r)$.
    Thus, we obtain
    \begin{align*}
        \varphi([y, h]) 
        = C_G(r_h) u_h v_y^{-1} H_y
        = C_G(r) (u_h g^{-1})^{-1} u_h v_y^{-1} v_y H_b
        = C_G(r) g H_b.
    \end{align*}
    This shows the map $\varphi$ is surjective.

    Finally, we show that $\varphi$ is injective.
    Assume $\varphi([x, g]) = \varphi([y, h])$, that is,
    we have $C_G(r_g)u_g v_x^{-1} H_{b_x} = C_G(r_h)u_h  v_y^{-1} H_{b_y}$ and $(b_x, r_g) = (b_y, r_h)$.
    Thus, there exist $\alpha \in C_G(r_g)$ and $\beta \in H_{b_x}$ such that $u_h v_y^{-1} = \alpha^{-1} u_g v_x^{-1} \beta$, 
    and hence $\gamma := u_g^{-1} \alpha u_h = v_x^{-1} \beta v_y$.
    Then, we have
    \begin{align*}
        (y, h) 
        &= (b_y^{v_y}, u_h^{-1} r_h u_h)
        = (b_x^{\beta v_y}, u_h^{-1} \alpha^{-1} r_g \alpha u_h)\\
        &= (x^{v_x^{-1} \beta v_y}, u_h^{-1} \alpha^{-1} u_g g u_g^{-1} \alpha u_h)
        = (x^{\gamma}, \gamma^{-1} g \gamma)
        = \Psi((x, g), \gamma).
    \end{align*}
    This shows that $[x, g] = [y, h]$ in $(X \times G)/G$, and we conclude that the map $\varphi$ is injective.
    This completes the proof. 
\end{proof}

We conclude the following theorem.

\begin{thm}\label{thm:pi_0}
   Let $\Phi: X \times G \to X$ be a right action.
   Let $B$ be a complete set of representatives for the orbit space $X/G$, 
   and let $R$ be a complete set of representatives for the conjugacy classes of $G$.
   \begin{enumerate}
    \item There exists a bijection between the set $\pi_0(X \times_\Phi G)$ of connected components and the disjoint union $\sqcup_{(b, r) \in B \times R} C_G(r) \backslash G / H_b$, where $C_G(r)$ is the centralizer of $r \in R$ in $G$, and $H_b$ is the stabilizer subgroup at $b \in B$ of the action $\Phi$.

    \item The connected component containing $(x, g) \in X \times_\Phi G$ corresponds to $C_G(r_g) u_g v_x^{-1} H_{b_x}$, where elements $r_g \in R$, $b_x \in B$, $u_g, v_x \in G$ are chosen so that
   \[
    g = u_g^{-1} r_g u_g, \quad x = b_x^{v_x}.
   \]
   \end{enumerate}
\end{thm}

\subsection{Subquandles}
\label{subsec:subqdle}
After describing the connected components, we next study another basic structural aspect of $X\times_\Phi G$, namely its subquandles.
In particular, we show that any subquandle of $X \times_\Phi G$ can be described in terms of a subgroup $\Gamma$ of $G$ and a subset $S$ of $X \times G$.
First, we introduce some notation.
We define the subset $C_\Gamma(S) \subset \Gamma$ by 
\[
    C_\Gamma(S) := \{g^{-1} a g \mid g \in \Gamma, a \in p_G(S)\},
\]
where $p_G: X \times G \to G$ is the natural projection.
Note that the second coordinates of elements obtained from $S$ by quandle operations lie in the conjugacy closure $C_\Gamma(S)$.
Let $D_S^\Gamma$ be the subgroup of $G$ generated by the subset $\{gh^{-1} \in G \mid g,h \in C_\Gamma(S)\}$.
Since $C_\Gamma(S)$ is invariant under conjugation by $\Gamma$, the subgroup $D_S^\Gamma$ is normal in $\Gamma$.
The subgroup $D^\Gamma_S$ is inspired by the displacement group of a quandle (cf. \cite{Hulpke-2016-ConnectedQuandlesTransitiveGroupsa}).
This subgroup records the products of these elements whose total exponent is zero.
\begin{lem}[cf. {\cite[Proposition 2.1]{Hulpke-2016-ConnectedQuandlesTransitiveGroupsa}}]\label{lem:D_S}
    The following equality holds:
    \[
    D_S^\Gamma = \left\{\alpha_1^{n_1} \cdots \alpha_k^{n_k} \in \Gamma \mid k \geq 0, n_i \in \{\pm 1\}, \alpha_i \in C_\Gamma(S), \sum_{i=1}^{k} n_i = 0\right\}.
    \]
\end{lem}
\begin{proof}
    Since $(gh^{-1})^{\pm 1}$ for $g, h \in C_\Gamma(S)$ is in the right-hand side,
    the inclusion $\subset$ holds.
    Let $d = \alpha_1^{n_1} \cdots \alpha_k^{n_k}$ be an element in the right hand side, that is, $\alpha_i \in C_\Gamma(S)$ and $n_i \in \{\pm 1\}$ with $\sum_i n_i = 0$.
    We prove the converse inclusion by induction on $k$.
    If $k = 0$, then $d = 1 \in D^\Gamma_S$, so the assertion holds.
    The case $k=1$ cannot occur because $\sum_{i} n_i = 0$.
    If $k = 2$, the element $d$ is equal to $\alpha_1 \alpha_2^{-1} \in D_S^\Gamma$ or $\alpha_1^{-1} \alpha_2 = (\alpha_1^{-1} \alpha_2 \alpha_1) \alpha_1^{-1} \in D_S^\Gamma$.

    We now consider the case $k \geq 3$.
    First, we assume $n_1 = n_k$.
    Then, there exists $1 < j < k$ such that $\sum_{i = 1}^{j} n_i = 0$.
    By the induction hypothesis, the elements $d_1 := \alpha_1^{n_1} \cdots \alpha_j^{n_j}$ and $d_2 := \alpha_{j+1}^{n_{j+1}} \cdots \alpha_k^{n_k}$ are in $D_S^\Gamma$.
    Thus, we have $d = d_1d_2 \in D_S^\Gamma$.

    Next, we assume $n_1 \neq n_k$.
    Then, the element $\delta := \alpha_2^{n_2} \cdots \alpha_{k-1}^{n_{k-1}}$ is in $D_S^\Gamma$ by the induction hypothesis.
    Since $\delta \in D^\Gamma_S$, the conjugate $\delta^{-1} \alpha_1 \delta$ belongs to $C_\Gamma(S)$.
    Since $n_1 \neq n_k$, the product 
    $\delta^{-1} \alpha_1^{n_1} \delta \alpha_k^{n_k}$ is of the form $c_1 c_2^{-1}$ or $c_1^{-1}c_2$ with $c_1, c_2 \in C_\Gamma(S)$, and hence belongs to $D_S^{\Gamma}$.
    As in the case $k = 2$, we obtain $(\delta^{-1} \alpha_1^{n_1} \delta) \alpha_k^{n_k} \in D_S^\Gamma$.
    Therefore,
    we have $d = \alpha_1^{n_1} \delta \alpha_k^{n_k} = \delta (\delta^{-1} \alpha_1^{n_1} \delta \alpha_k^{n_k}) \in D_S^\Gamma$.
    This completes the proof.
\end{proof}

\begin{dfn}
    Let $\Phi: X \times G \to X$ be a right action.
    For a subgroup $\Gamma$ of $G$ and a non-empty subset $S$ of $X \times \Gamma$, we define the subset $Q_\Phi(\Gamma; S)$ by
    \[
        Q_\Phi(\Gamma; S) := \{(b^d, d^{-1}ad) \mid (b, a) \in S, d \in D_S^\Gamma\}.
    \]
\end{dfn}

\begin{lem}\label{lem:Subquanlde_ActQdle}
    Let $\Phi: X \times G \to X$ be a right action.
    For a subgroup $\Gamma$ of $G$ and a non-empty subset $S$ of $X \times \Gamma$, the subset $Q_\Phi(\Gamma; S)$ of $X \times G$ is a subquandle of $X \times_\Phi G$.
\end{lem}

\begin{proof}
    Let $(x_1, g_1), (x_2, g_2) \in Q_\Phi(\Gamma;S)$.
    Then, there exist $(b_1, a_1), (b_2, a_2) \in S$, and $d_1, d_2 \in D_S^\Gamma$ such that $(x_1, g_1) = (b_1^{d_1}, d_1^{-1} a_1 d_1)$ and $(x_2, g_2) = (b_2^{d_2}, d_2^{-1} a_2 d_2)$.
    Then, we have 
    \begin{align*}
        (x_1, g_1) \qop_\Phi (x_2, g_2) = (x_1^{g_1^{-1}g_2}, g_2^{-1}g_1g_2) = (b_1^{d_1g_1^{-1}g_2}, (d_1g_1^{-1}g_2)^{-1}a_1(d_1g_1^{-1}g_2)).
    \end{align*}
    Since $d_1g_1^{-1}g_2 \in D_S^\Gamma$ by \cref{lem:D_S}, we obtain $(x_1, g_1) \qop_\Phi (x_2, g_2) \in Q_\Phi(\Gamma; S)$.
    Similarly, we can show $(x_1, g_1) \qop_\Phi^{-1} (x_2, g_2) \in Q_\Phi(\Gamma; S)$.
    This completes the proof.
\end{proof}

\begin{prop}
    Let $\Phi: X \times G \to X$ be a right action.
    Then, 
    any subquandle $Q$ of $X \times_\Phi G$ is equal to $Q_\Phi(\langle p_G(Q) \rangle; Q)$.
\end{prop}

\begin{proof}
    Set $S:=Q$.
    Let $\Gamma$ be the subgroup of $G$ generated by $p_G(Q)$.
    First, we show $C_\Gamma(S) = p_G(Q)$.
    The inclusion $\supset$ is immediate from the definition.
    To prove the inclusion $\subset$, let $\alpha \in C_\Gamma(S)$.
    Thus, there exist $a \in p_G(S)$ and $g \in \Gamma$ such that $\alpha = g^{-1} a g$.
    By the definition of $\Gamma$ and $S$, there exist $(x_0, a), (x_1, a_1), \dots, (x_n, a_n) \in Q$ and $\varepsilon_i \in \{\pm 1\}$ such that $g = a_1^{\varepsilon_1} \cdots a_n^{\varepsilon_n}$.
    By \cref{lem:formula}, we have 
    \[
        Q \ni ( \cdots ((x_0, a) \qop_\Phi^{\varepsilon_1} (x_1, a_1)) \qop_\Phi^{\varepsilon_2} \cdots )\qop_\Phi^{\varepsilon_n} (x_n, a_n)
        = (x_0^{a^{-(\varepsilon_1 + \cdots + \varepsilon_n)} g, g^{-1} a g}).
    \]
    Thus, we obtain $\alpha = g^{-1} a g \in p_G(Q)$, and hence $C_\Gamma(S) \subset p_G(Q)$.

    We now show that $Q = Q_\Phi(\Gamma; S)$.
    The inclusion $Q \subset Q_\Phi(\Gamma; S)$ is immediate from the definition.
    It remains to prove the reverse inclusion $\supset$.
    Let $(x, g) \in Q_\Phi(\Gamma; S)$.
    Thus, there exist $d \in D^\Gamma_S$ and $(b, a) \in S$ such that $(x, g) = (b^{d}, d^{-1} a d)$.
    By \cref{lem:D_S}, there exist $\alpha_1, \dots, \alpha_m \in C_\Gamma(S)$ and $\varepsilon_i \in \{\pm 1\}$ for $i = 1, \dots, m$ with $\sum_{i=1}^m \varepsilon_i = 0$ such that $d = \alpha_1^{\varepsilon_1} \cdots \alpha_m^{\varepsilon_m}$.
    Since $C_\Gamma(S) = p_G(Q)$, for every $i$, there exists $x_i \in X$ such that $(x_i, \alpha_i) \in Q$.
    Then, we have
    \[
        (x, g) = ( \cdots ((b, a) \qop_\Phi^{\varepsilon_1} (x_1, \alpha_1)) \qop_\Phi^{\varepsilon_2} \cdots )\qop_\Phi^{\varepsilon_m} (x_m, \alpha_m) \in Q,
    \]
    which completes the proof.
\end{proof}

We now consider the case where $S$ is a singleton.

\begin{prop}\label{prop:SpecialClass}
    Let $\Phi: X \times G \to X$ be a right action.
    Let $\Gamma$ be a subgroup of $G$, and let $S := \{(b, a)\} \subset X \times \Gamma$.
    Then, the following hold:
    \begin{enumerate}
        \item[$(1)$] If the action $\Phi$ is trivial, then $Q_\Phi(\Gamma; S)$ is isomorphic to $\Conj(a^{D_S^\Gamma})$, where $a^{D_S^\Gamma}$ is defined by 
        \[
          a^{D_S^\Gamma} = \{g^{-1}ag \mid g \in D_S^\Gamma\}.
        \]

        \item[$(2)$] If the action $\Phi$ is free, then $Q_\Phi(\Gamma; S)$ is isomorphic to $\GAlex(D_{S}^\Gamma, \iota_a)$, where $\iota_a: D_{S}^\Gamma \to D_{S}^\Gamma$ is defined by $\iota_a(g) := a^{-1}ga$.
    \end{enumerate}
\end{prop}

\begin{proof}
To show $(1)$, we prove that the restriction of the projection $p_G: X \times_\Phi G \to G$ is an isomorphism $f := p_G|_{Q_\Phi(\Gamma; S)}: Q_\Phi(\Gamma; S) \to \Conj(a^{D_S^\Gamma})$.
Since the map $p_G: X \times_\Phi G \to G$ is a homomorphism, the map $f$ is also a homomorphism.
The injectivity of the map $p_G$ follows from the triviality of the action $\Phi$.
By definition, $p_G(Q_\Phi(\Gamma; S)) = \Conj(a^{D_S^\Gamma})$.
Therefore, the map $f: Q_\Phi(\Gamma; S) \to \Conj(a^{D^\Gamma_S})$ is an isomorphism, as desired.

Next, we show $(2)$.
Note that $D_S^\Gamma$ is a normal subgroup of $\Gamma$.
Thus, the map $\iota_a$ restricts to an automorphism of $D_S^\Gamma$.
Define the map $h: \GAlex(D_{S}^\Gamma, \iota_a) \to Q_\Phi(\Gamma; S)$ by $h(d):=(b^d, d^{-1}a d)$ for $d \in D_{S}^\Gamma$.
By the definition of $Q_\Phi(\Gamma; S)$, the map $h$ is surjective.
Since the action $\Phi$ is free, the map $h$ is injective by comparing the first component.
For $d_1, d_2 \in \GAlex(D_{S}^\Gamma, \iota_a)$, we have 
\begin{align*}
    h(d_1 \qop d_2) 
    &= h(\iota_a(d_1 d_2^{-1})d_2)
    = (b^{\iota_a(d_1 d_2^{-1})d_2}, (\iota_a(d_1 d_2^{-1})d_2)^{-1} a \iota_a(d_1 d_2^{-1})d_2)\\
    &= (b^{a^{-1} d_1 d_2^{-1} a d_2}, 
    d_2^{-1} a^{-1}d_2 d_1^{-1} a d_1 d_2^{-1} a d_2)\\
    &= (b^{d_1 (d_1^{-1} a d_1)^{-1} d_2^{-1} a d_2}, (d_2^{-1} a d_2)^{-1} (d_1^{-1} a d_1)(d_2^{-1} a d_2))\\
    &= (b^{d_1}, d_1^{-1} a d_1) \qop_\Phi (b^{d_2}, d_2^{-1} a d_2)
    = h(d_1) \qop_\Phi h(d_2).
\end{align*}
Hence the map $h$ is a homomorphism.
Therefore, the map $h: \GAlex(D_{S}^\Gamma, \iota_a) \to Q_\Phi(\Gamma; S)$ is an isomorphism, which completes the proof.
\end{proof}

\begin{ex}[cf. {\cite[Lemma 25 and Theorem 33]{Eisermann-2003-HomologicalCharacterizationUnknota}}]\label{ex:longknotQdle}
	Let $G(K)$ be the knot group of a knot $K$ in $S^3 = \RR^3 \cup \{\infty\}$ with $\infty \in K$.
    Then, $\widetilde{K} := K \cap \RR^3$ is called the \emph{long knot} associated with $K$.
    Roughly speaking, the long knot is obtained by cutting the knot $K$ at the point $\infty$.
	We consider the action $\Phi: X \times G \to X$ of $G := G(K)$ on $X := G(K)$ by right multiplication, and set $S := \{(e, \mu)\}$, 
    where $\mu \in G(K)$ is a meridian of $K$.
	Then, the set $C_G(S)$ is the conjugacy class of $\mu$ in $G(K)$, 
    the subgroup $D_S^G$ is isomorphic to the commutator subgroup $[G(K), G(K)]$, 
    and the subquandle $Q_\Phi(G; S)$ of $X \times_\Phi G$ is isomorphic to the knot quandle of the associated long knot $\widetilde{K}$. 
\end{ex}

\begin{ex}[cf. {\cite[Section 5.1]{Eisermann-2014-QuandleCoveringsTheirGaloisa}}]\label{ex:UniversalCovering}
    Let $X$ be a connected quandle. 
    Recall that the \emph{associated group} $\mathrm{As}(X)$ of $X$ is the group defined by the following presentation:
    \[
        \mathrm{As}(X) = \langle w_x \, (x \in X) \mid w_{x \qop y} = w_y^{-1} w_x w_y \,(x, y \in X) \rangle.
    \]
    Let $\Phi_\mathrm{As}: \mathrm{As}(X) \times \mathrm{As}(X) \to \mathrm{As}(X)$ be the right regular action defined by $g^h := gh$ for $g, h \in \mathrm{As}(X)$.
    Fix $x \in X$, and let $S := \{(e, w_x)\}$.
    Then, the set $C_{\mathrm{As}(X)}(S)$ is equal to the set $\{w_y \mid y \in X\}$, 
    the subgroup $D_S^{\mathrm{As}(X)}$ is equal to the commutator subgroup $[\mathrm{As}(X), \mathrm{As}(X)]$,
    and the subquandle $Q_{\Phi_\mathrm{As}}(\mathrm{As}(X); S)$ is isomorphic to the universal quandle covering of the pointed quandle $(X, x)$.
\end{ex}

\subsection{Augmented quandles and natural embeddings}

We now relate the preceding subquandle construction to augmented quandles.
Recall that an \emph{augmented quandle}, introduced by Joyce \cite{Joyce-1982-ClassifyingInvariantKnotsKnota}, consists of a right action $\Phi: X\times G\to X$ together with a map $\varepsilon: X\to G$, called an \emph{augmentation}, satisfying
\[
  x^{\varepsilon(x)}=x,
  \qquad
  \varepsilon(x^g)=g^{-1}\varepsilon(x)g
\]
for all $x\in X$ and $g\in G$.
Such data induce a quandle structure on $X$ by
\[
  x \qop y = x^{\varepsilon(y)}.
\]
Conversely, for every quandle $X$, the action of $\Inn(X)$ on $X$ together with the map
\[
  \varepsilon_X: X \to \Inn(X),
  \qquad
  \varepsilon_X(x)=s_x,
\]
forms an augmented quandle, and the induced quandle structure coincides with the original one.
We define
\[
  \mathcal{I}_\Phi := \{(x, a) \in X \times G \mid x^a = x\}.
\]
It is easy to see that the set $\mathcal{I}_\Phi$ is invariant under the diagonal action $(x,a)^g := (x^g, g^{-1}ag)$ for $(x, a) \in X \times_\Phi G$ and $g \in G$.
Moreover, the projection $p_X: \mathcal{I}_\Phi \to X$ is a surjective $G$-equivariant map.
By \cref{prop:conn_comp}, 
the set $\mathcal{I}_\Phi$ is a union of connected components of $X \times_\Phi G$.

\begin{prop}\label{prop:char_aug}
Let $\Phi: X \times G\to X$ be a right action and let $\varepsilon: X\to G$ be a map.
Define
\[
  \sigma_\varepsilon\colon X\to X\times G,
  \qquad
  \sigma_\varepsilon(x):=(x,\varepsilon(x)).
\]
Then the following conditions are equivalent:
\begin{enumerate}
  \item[$(1)$] The map $\varepsilon$ is an augmentation.
  
  \item[$(2)$] The map $\sigma_\varepsilon$ is a $G$-equivariant section of $p_X: \mathcal I_\Phi\to X$.
\end{enumerate}
\end{prop}

\begin{proof}
First, we assume $(1)$.
From the definition, it is clear that $\sigma_\varepsilon$ is a section of $p_X$.
Let $x \in X$ and $g \in G$.
Since $\varepsilon$ is an augmentation, it satisfies $\varepsilon(x^g) = g^{-1}\varepsilon(x)g$.
Then, we have 
\begin{align*}
    \sigma_\varepsilon(x^g) 
    = (x^g, \varepsilon(x^g)) 
    = (x^g, g^{-1}\varepsilon(x)g)
    = (x, \varepsilon(x))^g
    = \sigma_\varepsilon(x)^g,
\end{align*}
and hence the map $\sigma_\varepsilon$ is $G$-equivariant.

Conversely, assume $(2)$.
Since $\sigma_\varepsilon$ is a section of $p_X: \mathcal{I}_\Phi \to X$, the element $\sigma_\varepsilon(x) = (x, \varepsilon(x))$ is in  $\mathcal{I}_\Phi$ for any $x \in X$.
This shows that $x^{\varepsilon
(x)} = x$.
Moreover, since $\sigma_\varepsilon$ is $G$-equivariant, for any $x \in X$ and $g \in G$, it satisfies
\[
   (x^g, \varepsilon(x^g)) 
    = \sigma_\varepsilon(x^g)
    = \sigma_\varepsilon(x)^g
    = (x, \varepsilon(x))^g
    = (x^g, g^{-1}\varepsilon(x)g).
\]
Thus, we have $\varepsilon(x^g) = g^{-1}\varepsilon(x)g$.
Therefore, we conclude that $\varepsilon$ is an augmentation, which completes the proof.
\end{proof}

\begin{prop}\label{prop:aug_emb}
Let $\Phi: X \times G \to X$ be a right action, and let $\varepsilon: X \to G$ be an augmentation.
Equip $X$ with the quandle structure induced by $\varepsilon$.
Then the following statements hold:
\begin{enumerate}
  \item[$(1)$] The map $\sigma_\varepsilon: X \longrightarrow X\times_\Phi G$ defined by $x \longmapsto (x,\varepsilon(x))$ is an injective quandle homomorphism.

  \item[$(2)$] $\mathrm{Graph}(\varepsilon) := \{(x,\varepsilon(x))\mid x\in X\}$ is a subquandle of $X \times_\Phi G$.

  \item[$(3)$] $\mathrm{Graph}(\varepsilon)$ is invariant under the diagonal $G$-action.
\end{enumerate}
\end{prop}
\begin{proof}
  We show $(1)$.
  It is clear that the map $\sigma_\varepsilon$ is injective.
  Recall that the point symmetry at $(x, g) \in X \times_\Phi G$ is denoted by $s_g:=s_{(x,g)}$.
  Then, it satisfies
  \begin{align*}
    \sigma_\varepsilon(x)^{\varepsilon(y)}
    = (x, \varepsilon(x))^{\varepsilon(y)}
    = (x^{\varepsilon(y)}, \varepsilon(y)^{-1}\varepsilon(x)\varepsilon(y))
    = \sigma_\varepsilon(x)^{s_{\varepsilon(y)}}
  \end{align*}
  for any $x, y \in X$.
  Since $\sigma_\varepsilon$ is $G$-equivariant by \cref{prop:char_aug}, we have
  \begin{align*}
      \sigma_\varepsilon(x \qop y)
      = \sigma_\varepsilon(x^{\varepsilon(y)})
      = \sigma_\varepsilon(x)^{\varepsilon(y)}
      = \sigma_\varepsilon(x)^{s_{\varepsilon(y)}}
      = \sigma_\varepsilon(x)^{s_{\sigma_\varepsilon(y)}}
      = \sigma_\varepsilon(x) \qop_\Phi \sigma_\varepsilon(y)
  \end{align*}
  for any $x, y \in X$.
  Thus, the assertion $(1)$ holds.

  Since $\mathrm{Graph}(\varepsilon)$ is the image of the $G$-equivariant quandle homomorphism $\sigma_\varepsilon$,
  assertions $(2)$ and $(3)$ immediately follow, which completes the proof.
\end{proof}

The following proposition describes 
$\mathrm{Graph}(\varepsilon)$ in terms of the construction introduced in \cref{subsec:subqdle}.

\begin{prop}\label{prop:aug_subqdle}
Let $\Phi: X\times G\to X$ be a right action and let $\varepsilon: X\to G$ be an augmentation.
Let $H_\varepsilon$ be the subgroup of $G$ generated by $\varepsilon(X)$.
Let $B \subset X$ be a complete set of representatives for the $H_\varepsilon$-orbits in $X$, and put
\[
  S_\varepsilon
  := \{(b,\varepsilon(b))\mid b\in B\} \subset X\times H_\varepsilon.
\]
Then we have
  \[
    C_{H_\varepsilon}(S_\varepsilon)
    = \varepsilon(X), \qquad
    D^{H_\varepsilon}_{S_\varepsilon} = 
    \left\langle
      \varepsilon(x)\varepsilon(y)^{-1} \mid x,y\in X
    \right\rangle, \qquad
    \mathrm{Graph}(\varepsilon)
    = Q_\Phi(H_\varepsilon;S_\varepsilon).
  \]
\end{prop}
\begin{proof}
    First, we show that the equality $C_{H_\varepsilon}(S_\varepsilon) = \varepsilon(X)$ holds.
    The inclusion $\subset$ is clear.
    We prove the converse inclusion.
    Let $g \in \varepsilon(X)$.
    Thus, there exists $x \in X$ such that $g = \varepsilon(x)$.
    Since $B$ is a complete set of representatives for $H_\varepsilon$-orbits, 
    there exist $b \in B$ and $h \in H_\varepsilon$ such that $x = b^h$.
    Then, we have $(b, \varepsilon(b)) \in S_\varepsilon$ and
    \[
        (x, g) 
        = (b^h, \varepsilon(b^h)) 
        = (b^h, h^{-1}\varepsilon(b)h).
    \]
    Therefore, we obtain $g = h^{-1}\varepsilon(b)h \in C_{H_\varepsilon}(S_\varepsilon)$, as desired.

    The equality $D^{H_\varepsilon}_{S_\varepsilon} = 
    \left\langle
      \varepsilon(x)\varepsilon(y)^{-1} \mid x,y\in X
    \right\rangle$ 
    follows directly from the definition.
    We now show 
    \[
        \mathrm{Graph}(\varepsilon)
        = Q_\Phi(H_\varepsilon;S_\varepsilon).
    \]
    Let $(x, \varepsilon(x)) \in \mathrm{Graph}(\varepsilon)$ for $x \in X$.
    Then, there exist $b \in B$ and $h \in H_\varepsilon$ such that $x = b^h$.
    From \cref{lem:D_S}, there exists $k \in \ZZ$ such that $g := \varepsilon(b)^kh \in D^{H_\varepsilon}_{S_\varepsilon}$.
    \[
        (x, \varepsilon(x)) = (b^h, h^{-1} \varepsilon(b) h)  = (b^g, g^{-1} \varepsilon(b) g) \in Q_\Phi(H_\varepsilon;S_\varepsilon).
    \]
    Conversely, every element of $Q_\Phi(H_\varepsilon; S_\varepsilon)$ is of the form 
    $(b^h, \varepsilon(b^h)) = (b^h, h^{-1}\varepsilon(b)h)$, and hence belongs to $\mathrm{Graph}(\varepsilon)$.
    Thus, $Q_\Phi(H_\varepsilon; S_\varepsilon) \subset \mathrm{Graph}(\varepsilon)$, 
    which completes the proof.
\end{proof}

Recall the action $\Phi_{\mathrm{Inn}} \colon X\times\mathrm{Inn}(X)\to X$ of the inner automorphism group of a quandle.
By \cref{prop:aug_emb}, 
we obtain the following natural embedding.

\begin{dfn}\label{def:natural_embedding}
Let $X$ be a quandle.
The embedding  
\[
  \varphi_X\colon
  X\longrightarrow
  X\times_{\Phi_{\mathrm{Inn}}}\mathrm{Inn}(X),
  \qquad
  x\longmapsto(x,s_x)
\]
corresponding to the augmentation $s: X \to \Inn(X)$ is called the \emph{natural embedding} of $X$.
\end{dfn}

\cref{prop:aug_subqdle} also yields the following corollary.

\begin{cor}
Let $X$ be a quandle, let $B\subset X$ be a base point set, and 
let $S:=\varphi_X(B) \subset X\times\mathrm{Inn}(X)$.
Then, we have
\[
  C_{\mathrm{Inn}(X)}(S)
  = \{s_x\mid x\in X\}, \qquad
  D^{\mathrm{Inn}(X)}_S
  = \mathrm{Dis}(X), \qquad
  \varphi_X(X)
  =
  Q_{\Phi_{\mathrm{Inn}}}
  \left(\mathrm{Inn}(X);S\right)
  \cong X.
\]
\end{cor}



\section{The Cayley-type embedding theorem}

In this section, we apply the natural embedding constructed above to prove the main theorem of the paper.
For a positive integer $n$,
we denote $[n] := \{1, \dots, n\}$.
Let $\Phi_n: [n] \times \SS_n \to [n]$ be the right action of $\SS_n$ on $[n]$ as permutations.

\begin{dfn}\label{def:QS_n}
  The quandle $[n] \times_{\Phi_n} \SS_n$ obtained from the right action $\Phi_n: [n] \times \SS_n \to [n]$ is denoted by $Q\SS_n$.
\end{dfn}

In quandle theory, the quandle $Q\SS_n$ plays a role analogous to that of $\SS_n$ in group theory.
In fact, a Cayley-type embedding theorem holds as follows.

\begin{thm}\label{thm:CayEmbThm}
  Let $X$ be a quandle of cardinality $n$.
  Then, there exists an injective quandle homomorphism $X \to Q\SS_n$.
\end{thm}

\begin{proof}
  Write $X = \{x_1, \dots, x_n\}$.
  Let $f: X \to [n]$ be the map defined by $f(x_i) := i$.
  Then, this map induces the permutation representation $\rho: \Inn(X) \to \SS_n$ by \[
    f(x^g) = f(x)^{\rho(g)}
  \]
  for $x \in X$ and $g \in \Inn(X)$.
  Then, the pair $(f, \rho)$ is a morphism $\Phi_\Inn \to \Phi_n$ of right actions.
  By \cref{prop:EquivariantHom}, the map $f \times \rho: X \times_{\Phi_\Inn} \Inn(X) \to Q\SS_n$ is a homomorphism.
  Recall that the natural embedding $\varphi_X: X \to X \times_\Phi \Inn(X)$ is an injective homomorphism; see \cref{prop:aug_emb} and \cref{def:natural_embedding}.
  Thus, the composition $(f \times \rho) \circ \varphi_X: X \to Q\SS_n$ is a quandle homomorphism.
  Since $(f \times \rho) \circ \varphi_X(x_i) = (i, \rho(s_{x_i}))$, this map is injective, which completes the proof.
\end{proof}

Note that the order of the map $\zeta_1: Q\SS_n \to Q\SS_n$ is equal to the least common multiple of $\{1, \dots, n\}$.
Since the action $\Phi_n: [n] \times \SS_n \to [n]$ is faithful,
we can determine the inner automorphism group of $Q\SS_n$ from \cref{cor:InnerAutomorphismGroup2}.
\begin{prop}\label{prop:Inn_QSn}
  $\Inn(Q\SS_n) \cong \ZZ/l\ZZ \times \SS_n$, where $l$ is the least common multiple of $\{1, \dots, n\}$.
\end{prop}

To describe the connected components of $Q\SS_n$, we recall some standard terminology for the symmetric group $\SS_n$.
A \emph{partition} $\lambda$ of $n$ is a sequence $(m_1, \dots, m_n)$ of non-negative integers such that $\sum_{k=1}^n k m_k = n$,
and then we denote $\lambda = (1^{m_1}2^{m_2} \cdots n^{m_n})$ and $\lambda \vdash n$.
The decomposition of $\sigma \in \SS_n$ into a product of disjoint cycles is called the \emph{cycle decomposition} of $\sigma$.
The cycle decomposition is unique up to reordering the disjoint cycles and cyclically permuting the entries within each cycle.
The \emph{cycle type} of $\sigma$ is the partition $\lambda(\sigma) := (1^{m_1}2^{m_2} \cdots n^{m_n})$, where $m_k$ is the number of cycles of length $k$ in the cycle decomposition.
For $i \in [n]$, we denote by $d(i, \sigma)$ the length of the cycle containing $i$ in the cycle decomposition of $\sigma$.
Then, the following holds.

\begin{prop}\label{prop:conn_QSn}
Let $(i,\sigma),(j,\tau) \in Q\SS_n$. Then the following conditions are equivalent:
\begin{enumerate}
  \item There exists $\alpha\in \mathcal{S}_n$ such that
    \[
    j=i^\alpha
    \quad\text{and}\quad
    \tau=\alpha^{-1}\sigma\alpha.
    \]

  \item $\lambda(\sigma) = \lambda(\tau)$ and $d(i, \sigma) = d(j, \tau)$.

  \item $(i, \sigma)$ and $(j, \tau)$ belong to the same connected component in $Q\SS_n$.
\end{enumerate}
\end{prop}

\begin{proof}
The equivalence of assertions $(1)$ and $(3)$ follows from \cref{prop:conn_comp}.
Here, we show that $(1)$ and $(2)$ are equivalent.
First, assume $(1)$.
Thus, there exists $\alpha\in\mathcal{S}_n$ such that
$j=i^\alpha$ and $\tau=\alpha^{-1}\sigma\alpha$.
Since $\sigma$ and $\tau$ are conjugate in $\SS_n$, they have the same cycle type, that is, $\lambda(\sigma) = \lambda(\tau)$.
Let $d=d(i,\sigma)$. Then, the elements $i,\ i^\sigma,\ \ldots,\ i^{\sigma^{d-1}}$
are pairwise distinct and $i^{\sigma^d}=i$. Since
$\tau=\alpha^{-1}\sigma\alpha$, we have $(i^\alpha)^{\tau^k}=(i^{\sigma^k})^\alpha$ for every $k\geq 0$. Hence, the elements $j,\ j^\tau,\ \ldots,\ j^{\tau^{d-1}}$
are pairwise distinct, while
\[
    j^{\tau^d} = (i^\alpha)^{\tau^d} = (i^{\sigma^d})^\alpha = i^\alpha = j.
\]
Therefore, we have $d(j,\tau)=d(i,\sigma)$.

Conversely, assume $(2)$.
Let $d := d(i,\sigma)=d(j,\tau)$.
Write the cycle of $\sigma$ containing $i$ as
$C=(i,i^\sigma,\ldots,i^{\sigma^{d-1}})$
and the cycle of $\tau$ containing $j$ as
$D=(j,j^\tau,\ldots,j^{\tau^{d-1}})$.
Define a bijection from $C$ to $D$ by
\[
(i^{\sigma^k})^\alpha:=j^{\tau^k},
\qquad
0\leq k<d.
\]
In particular, $i^\alpha=j$.
Since $\lambda(\sigma) = \lambda(\tau)$, after removing $C$ from the cycle decomposition of $\sigma$ and $D$ from that of $\tau$, the remaining cycles can be paired so that the corresponding cycles have the same length. 
Let $C_1,\ldots,C_r$ be the remaining cycles of $\sigma$, and let $D_1,\ldots,D_r$ be the remaining cycles of $\tau$, indexed so that
$|C_l|=|D_l|$ for every $l$.
For each $\ell$, write
$C_l
=
(a_l,a_l^\sigma,\ldots,
a_l^{\sigma^{m_l-1}})$
and
$D_l
=
(b_l,b_l^\tau,\ldots,
b_l^{\tau^{m_l-1}})$,
where $m_l=|C_l|=|D_l|$.
Define
\[
(a_l^{\sigma^k})^\alpha
:=
b_l^{\tau^k},
\qquad
0 \leq k<m_l.
\]
Together with the definition of $\alpha$ on $C$, this determines a permutation $\alpha \in \SS_n$.
By construction, for every $x\in[n]$, it satisfies $(x^\sigma)^\alpha=(x^\alpha)^\tau$.
Hence, we have $\sigma\alpha=\alpha\tau$,
and therefore $\tau=\alpha^{-1}\sigma\alpha$.
Thus the assertion $(1)$ holds.
This completes the proof.
\end{proof}

For a partition $\lambda=(1^{m_1}2^{m_2}\cdots n^{m_n})$ of $n$,
we define
\[
\mathrm{supp}(\lambda) := \{d\in [n]\mid m_d>0\}.
\]
For $d\in\mathrm{supp}(\lambda)$, let
\[
\Omega_{\lambda,d} :=
\left\{
(i,\sigma)\in [n]\times \SS_n \mid \lambda(\sigma)=\lambda
\text{ and }
d(i,\sigma)=d
\right\}.
\]
By \cref{prop:conn_QSn}, we obtain the following description of $\pi_0(Q\mathcal{S}_n)$.

\begin{cor}\label{cor:conn_QSn}
The connected components of $Q\mathcal{S}_n$ are precisely the subsets $\Omega_{\lambda,d}$ for a partition $\lambda \vdash n$ and $d \in \mathrm{supp}(\lambda)$.
In particular, there is a bijection
\[
\pi_0(Q\mathcal{S}_n)
\longrightarrow
\sqcup_{\lambda\vdash n}\mathrm{supp}(\lambda)
\]
given by
$[(i,\sigma)] \longmapsto \bigl(\lambda(\sigma),d(i,\sigma)\bigr)$.
\end{cor}

We conclude with the linear analogue announced in the introduction.
Let $V$ be an $n$-dimensional vector space over a field $\mathbb{K}$, and fix a basis $\{e_i \mid i \in [n]\}$ of $V$.
We identify $V$ with the space of row vectors $\mathbb{K}^{n}$. 
Under this identification, we let $GL(V)$ act on $V$ from the right by matrix multiplication:
\[
v^A:=vA
\qquad (v\in V,\ A\in GL(V)).
\]
We denote this right action by $\Psi_V: V \times GL(V)\to V$.

\begin{dfn}
  The quandle $V \times_{\Psi_V} GL(V)$ obtained 
  from the right action $\Psi_V: V \times GL(V) \to V$ is denoted by $GLQ(V)$.
\end{dfn}

Let $\rho: \SS_n \to GL(V)$ be the permutation representation.
More precisely, for $\sigma \in \SS_n$, the linear map $\rho(\sigma): V \to V$ is defined by $e_i \rho(\sigma) = e_{\sigma(i)}$.
We define a map $e: [n] \to V$ by $e(i) := e_i$ for $i \in [n]$.
Then, the pair $(e, \rho)$ is a morphism $\Phi_n \to \Psi_V$, and the map $e \times \rho: Q\SS_n \to V \times_{\Psi_V} GL(V)$ is injective.
Thus, we have the following theorem.

\begin{thm}\label{thm:LinearRep}
  For any field $\mathbb{K}$, every quandle of cardinality $n$ embeds into $GLQ(\mathbb{K}^n)$.
\end{thm}

\begin{rmk}
  The notion of a representation of  a quandle has been studied in \cite{Joyce-1982-ClassifyingInvariantKnotsKnota,Eisermann-2003-HomologicalCharacterizationUnknota}.
  A representation of a quandle in a group $G$ is a quandle homomorphism into $\Conj(G)$.
  Linear representations of racks and quandles were studied by Elhamdadi and Moutuou \cite{Elhamdadi-2018-FinitelyStableRacksRacka}, and, more recently, by Maassarani \cite{Maassarani-2026-IrreducibleRepresentationsQuandlesa}.
  In group theory, every finite group admits a finite-dimensional faithful linear representation.
  In contrast, there exist finite quandles that admit no injective quandle homomorphism into $\Conj(GL(V))$ for any finite-dimensional vector space $V$.
  For example, let $Q_8 = \{\pm 1, \pm i, \pm j, \pm k\}$ be the quaternion group and let $\sigma \in \mathrm{Aut}(Q_8)$ be defined by 
  \[
    \sigma(i) = j, \quad \sigma(j) = k, \quad \sigma(k) = i.
  \]
  Then, the generalized Alexander quandle $\GAlex(Q_8, \sigma)$ cannot be embedded into any conjugation quandle (cf. \cite[Corollary 3.5 and Proposition 4.8]{Arai-2026-DetectingNonadmissibilityQuandlesColoringsa}), and hence admits no injective quandle homomorphism into $\Conj(GL(V))$ for any vector space $V$.

  On the other hand, the map 
  \[
    \Conj(GL(V)) \to GLQ(V), \quad A \mapsto (0, A)
  \]
  is an injective quandle homomorphism.
  Thus, $\Conj(GL(V))$ can be identified with the subquandle $\{0\} \times GL(V)$ of $GLQ(V)$.
  This shows that quandle homomorphisms into $GLQ(V)$ may be regarded as a generalization of linear representations of quandles. 
  In particular, \cref{thm:LinearRep} shows that every finite quandle admits an injective quandle homomorphism into $GLQ(V)$ for some finite-dimensional vector space $V$, even though it need not admit a faithful linear representation into $\Conj(GL(V))$.
\end{rmk}

\section*{Acknowledgment}
The author would like to thank Katsunori Arai and Hiroshi Tamaru for helpful comments and useful discussions. 
This work was partly supported by MEXT Promotion of Distinctive Joint Research Center Program JPMXP0723833165 and Osaka Metropolitan University Strategic Research Promotion Project (Development of International Research Hubs).
\bibliographystyle{alphaurl}
\bibliography{actqdle}

	
\end{document}